\documentclass[a4paper, 12pt, oneside, reqno, notitlepage]{amsart}
\usepackage[margin=3cm]{geometry}
\usepackage[utf8]{inputenc}
 \usepackage[T1]{fontenc}
\usepackage[usenames,dvipsnames]{color}
\usepackage[colorlinks=true,linkcolor=Red,citecolor=Green]{hyperref}
\usepackage[super]{nth}
\usepackage[open, openlevel=2, depth=3, atend]{bookmark}
\hypersetup{pdfstartview=XYZ}
\usepackage[font=footnotesize]{caption}
\usepackage{a4wide}
\usepackage[all]{xy}
\usepackage{enumitem}

\scrollmode
\usepackage{latexsym}
\def\.{\cdot}

\def\e{\mathbf{e}}

\def\beq{\begin{equation}}
\def\eeq{\end{equation}}
\def\bea{\begin{eqnarray*}}
\def\eea{\end{eqnarray*}}
\def\beaa{\begin{eqnarray}}
\def\eeaa{\end{eqnarray}}
\def\ba{\begin{array}}
\def\ea{\end{array}}
\def\f{\varphi}

\def\L{\Lambda}

\def \RM{\mathbb{R}}

\def \ZM{\mathbb{Z}}

\def \HH{\mathbb{H}}

\def\V{\mathbb{V}}

\newcommand{\mc}{\mathcal}
\newcommand{\eps}{\varepsilon}
\newcommand{\Ss}{\mathbb{S}}
\newcommand{\X}{\mathbf{X}}

\def\id{\mathrm{id}}
\def\be{\begin{equation}}
\def\ee{\end{equation}}

\def\Aut{\mathrm{Aut }}

\def\Hom{\mathrm{Hom}}
\def\Sym{\mathrm{Sym}}

\def\so{\mathfrak{so}}

\def\R{\mathrm{R}}

\def\SO{\mathrm{SO}}
\def\OO{\mathrm{O}}
\def\End{\mathrm{End}}

\def\Sym{\mathrm{Sym}}

\def\Id{\mathrm{id}}

\def\T{\mathrm{T}}

\def\dd{\mathrm{d}}
\def\LL{\mathrm{L}}

\def\RO{R}
\def\p{\mathrm{p}}
\def\q{\mathrm{q}}
\def\dd{\mathrm{D}}

\newtheorem{pro}{Proposition}[section]

\newtheorem{lemma}[pro]{Lemma}

\theoremstyle{definition}

\title[Pestov and Weitzenb\"ock identities]{Correspondence between Pestov and Weitzenb\"ock identities}

\author[M. Ceki\'c]{Mihajlo Ceki\'c}
\address{Mihajlo Ceki\'c, Institut f\"ur Mathematik, Universit\"at Z\"urich, Winterthurerstrasse 190, CH-8057 Z\"urich, Switzerland}
\email{mihajlo.cekic@math.uzh.ch}

\author[T. Lefeuvre]{Thibault Lefeuvre}
\address{Thibault Lefeuvre, Université de Paris and Sorbonne Université, CNRS, IMJ-PRG, F-75006 Paris, France.}
\email{tlefeuvre@imj-prg.fr}

\author[A. Moroianu]{Andrei Moroianu}
\address{Andrei Moroianu, Université Paris-Saclay, CNRS,  Laboratoire de mathématiques d'Orsay, 91405, Orsay, France}
\email{andrei.moroianu@math.cnrs.fr}

\author[U. Semmelmann]{Uwe Semmelmann}
\address{Uwe Semmelmann, Institut f\"ur Geometrie und Topologie, Fachbereich Mathematik, Universit{\"a}t Stuttgart, Pfaffenwaldring 57, 70569 Stuttgart, Germany}
\email{uwe.semmelmann@mathematik.uni-stuttgart.de}

\date{\today}

\begin{document}


\begin{abstract}
The aim of this note is to establish the correspondence between the twisted localized Pestov identity on the unit tangent bundle of a Riemannian manifold and the Weitzenb\"ock identity for twisted symmetric tensors on the manifold.
\end{abstract}
\maketitle

\section{Introduction}

The Weitzenb\"ock and Pestov identities are two standard identities in Riemannian geometry. While the former is usually phrased on the base manifold, the Pestov identity is given in terms of functions on the unit tangent bundle. The latter can be further \emph{localized} by considering specific functions which are spherical harmonics in restriction to every fiber of the unit tangent bundle: this is known as the localized Pestov identity. There is a tautological correspondence between trace-free symmetric tensors on the base manifold and spherical harmonics; hence, it is conceivable that the Weitzenböck identity \emph{should} be related to the localized Pestov identity but this correspondence has never been established anywhere formally. The purpose of this note is therefore to show that the localized Pestov identity is indeed equivalent to the Weitzenböck identity. More generally, we will consider this correspondence for twisted objects, where we twist by an auxiliary vector bundle over the Riemannian manifold. As both identities require a certain amount of notation before being stated, we refer the reader to Proposition \ref{tw} below for the twisted Weitzenböck identity, and Proposition \ref{proposition:plpi} for the twisted localized Pestov identity. As for the introduction, we provide a brief account on the history of these identities, and for which purposes they are used.

The Pestov identity is an $L^2$ energy identity on the unit tangent bundle of a Riemannian manifold which was first introduced by Mukhometov \cite{Mukhometov-75,Mukhometov-81} and Amirov \cite{Amirov-86}, then in a more general form by Pestov and Sharafutdinov \cite{Pestov-Sharafutdinov-88, Sharafutdinov-94}, and later written in an intrinsic way by Knieper \cite{Knieper-02}\footnote{We also remark that in \cite[Appendix]{Knieper-02}, Knieper argues that the Pestov identity is a ``formula of Weitzenböck type''. Somehow, the present paper makes this intuition rigorous.}. More recently, a \emph{twisted} identity (that is, involving an auxiliary vector bundle) was obtained by Guillarmou, Paternain, Salo and Uhlmann \cite{Paternain-Salo-Uhlmann-15,Guillarmou-Paternain-Salo-Uhlmann-16}. The Pestov identity was found to play an essential role in two problems of Riemannian geometry on negatively-curved manifold, namely:
\begin{enumerate}[label=(\roman*), itemsep=5pt]
    \item the \emph{marked length spectrum rigidity} problem which consists in recovering a metric from the knowledge of the lengths of its closed geodesics (marked by the free homotopy of the manifold). Equally important and intimately related are the \emph{tensor tomography question} which asks to recover a tensor from its integrals along closed geodesics, and \emph{inverse spectral problems}, which ask if the spectrum of a geometric operator determines the geometry; see \cite{Guillemin-Kazhdan-80, Croke-Sharafutdinov-98,Paternain-Salo-Uhlmann-13, Paternain-Salo-Uhlmann-14-1, Guillarmou-Lefeuvre-18} for references where the Pestov identity is used; see also \cite{Croke-90,Otal-90,Croke-Fathi-Feldman-92,Hamenstadt-99} for further references on the marked length spectrum.
    \item the \emph{ergodicity of the frame flow} which consists in showing that the only measurable functions that are invariant by the frame flow on the frame bundle are the constant functions, see \cite{Cekic-Lefeuvre-Moroianu-Semmelmann-21, Cekic-Lefeuvre-Moroianu-Semmelmann-23,Cekic-Lefeuvre-22} for references where the Pestov identity is used; see also \cite{Brin-Gromov-80,Brin-Karcher-84,Burns-Pollicott-03} for further references on frame flow ergodicity.
\end{enumerate}

Let us also mention that there are other versions of the Pestov identity related to \emph{thermostat flows} \cite{Jane-Paternain-09}, and that the (localized) twisted Pestov identity for \emph{non-metric} connections can be improved using Carleman estimates \cite{Paternain-Salo-22}.

The Weitzenb\"ock formula usually expresses a curvature term as a linear combination of operators of the form $P^*P$, where $P$ is a first-order differential operator, typically a projection of the covariant derivative. It
is an important tool for combining differential geometric aspects with topological aspects on compact Riemannian manifolds, see \cite{Bourguignon-90} for a nice review. This is prominently illustrated in the Bochner method, where the vanishing of Betti numbers follows under suitable curvature assumptions, and also for the non-existence of metrics of positive scalar curvature on spin manifolds with non-vanishing $\hat A$-genus. Moreover, it is used to prove eigenvalue estimates for Laplace and Dirac type operators. 

In this note we give a self-contained proof of the Weitzenb\"ock formula on trace-free symmetric tensors. This is a special case of a more general method
introduced in \cite{Semmelmann-Weingart-10}. Here we will show in addition how to extend
the Weitzenb\"ock formula to the case of symmetric tensors twisted with an auxiliary vector bundle $E$. Finally, we show that this twisted Weitenzenb\"ock formula translates into the localized twisted Pestov identity on the unit tangent bundle. \\

\noindent \textbf{Acknowledgement:} The authors wish to thank the CIRM, where part of this article was written, for support and  hospitality. M.C. acknowledges the support of an Ambizione grant (project number 201806) from the Swiss National Science Foundation. We would also like to thank G.~Knieper for pointing out the reference \cite{Knieper-02} which was missing in an earlier version of this article; we thank G.~Paternain, M.~Salo, and G.~Uhlmann for their comments.

\section{Symmetric tensors}

In this section we recall basic formulas for symmetric tensors as well as the definition and first properties of conformal Killing tensors. More details can be
found in \cite{Heil-Moroianu-Semmelmann-16}.

\subsection{The symmetric algebra of a vector space}
Let $(\T, g):=\R^n$ be the standard Euclidean vector space of dimension $n$. We denote with
$\Sym^k \T \subset \T^{\otimes k}$ the $k$-fold symmetric tensor product of $\T$. Elements of $\Sym^k \T$
are symmetrized tensor products
\begin{equation}
    \label{equation:product}
v_1 \cdot \ldots \cdot   v_k := \sum_{\sigma \in S_k} \, v_{\sigma(1)} \otimes \ldots \otimes v_{\sigma(k)},
\end{equation}
where $v_1, \ldots, v_k$ are vectors in $\T$. In particular we have $v\cdot u = v\otimes u + u \otimes v$ for $u, v \in \T$. Some authors (see \cite[page 156]{Paternain-Salo-Uhlmann-23}) use another convention for the symmetric product and divide by $k!$ in \eqref{equation:product}.

Using the metric $g$, one can identify $\T$ with $\T^*$. Under this identification, $g\in\Sym^2 \T^*\simeq\Sym^2 \T$ can be written as $g=\tfrac12  \sum_i \e_i \cdot \e_i$, for any
orthonormal basis $\{ \e_i\}$. The direct sum $\Sym\,\T:=\bigoplus_{k\ge 0} \Sym^k \T$ is endowed with a commutative product making $\Sym\, \T$ into a $\ZM$-graded commutative algebra.
The scalar product $g$ induces a scalar product on $\Sym^k \T$, also denoted by $g$, defined by 
$$
g(v_1 \cdot \ldots \cdot v_k,w_1 \cdot \ldots \cdot w_k) \;=\; \sum_{\sigma \in S_k}
\, g(v_1, w_{\sigma(1)}) \cdot \ldots  \cdot g(v_k, w_{\sigma(k)}) .
$$
With respect to this scalar product, every element $K$ of $\Sym^k \T$ can be identified with a symmetric $k$-linear map (i.e. a polynomial of degree $k$) on $\T$ by the formula 
$$K(v_1,\ldots,v_k)=g(K,v_1\cdot\ldots\cdot v_k) .
$$
For every $v\in \T$, the metric adjoint of the linear map
$
v\cdot: \Sym^k \T \rightarrow \Sym^{k+1} \T, \;  K \mapsto v \cdot K
$
is the contraction
$
v\lrcorner : \Sym^{k+1} \T \rightarrow \Sym^{k} \T, \;  K \mapsto v \lrcorner \, K
$, defined by $(v \lrcorner \, K) (v_1, \ldots , v_{k-1}) = K(v, v_1, \ldots, v_{k-1})$.
In particular we have $v \lrcorner \, u^k = k g(v, u) u ^{k-1},$ for all $  v, u \in \T$.

We introduce the linear map $\, \deg : \Sym\, \T \rightarrow \Sym\, \T$, defined by
$\deg (K)  = k K$ for $ K \in \Sym^k \T$. Then we have\;
$$
\sum_i \e_i \cdot \e_i \lrcorner \, K = \deg(K), \quad \sum_i \e_i \lrcorner \e_i \cdot K = n K + \deg(K),
$$
where $\{\e_i\}$ denotes an orthonormal frame of $(\T, g)$. Note that if $K\in \Sym^k \T$
is considered as a polynomial of degree $k$ then $v \lrcorner K$ corresponds to
the directional derivative $\partial_v K$ and the last formula is nothing else than the well-known Euler formula on homogeneous functions.

Contraction and multiplication with the symmetric tensor $\LL:=\sum_i \e_i\cdot \e_i = 2g$ defines two  additional  linear maps:
$$
\L : \Sym^k \T \rightarrow \Sym^{k-2} \T , \quad K \mapsto  \sum_i \e_i \lrcorner \,  \e_i \lrcorner \,  K 
$$
and
$$
\LL  : \Sym^{k-2} \T \rightarrow \Sym^{k} \T , \quad K \mapsto  \LL \cdot K ,
$$
which are adjoint to each other.
It is straightforward to check the following algebraic commutator relations
\beq\label{commu}
[\, \Lambda, \, \LL  \,] \;=\; 2n\,  \Id \;+\;4 \deg,\quad [\, \deg, \LL  \,] = 2\, \LL , \quad [\,\deg, \L\,] = -\,2 \,\Lambda  ,
\eeq
and for every $v\in T$:
\beq \label{commu2}
[\, \Lambda, \, v \,\cdot \, ] \;=\; 2\, v \,  \lrcorner  \, , \quad
[\, v \lrcorner \,,\, \LL  \,  \,] \;=\; 2\, v \cdot\, ,\quad
[\,\L , \, v \lrcorner \, \,] \;=\; 0 \;=\; [\, \LL , \, v \cdot \,]  . 
\eeq

For $\T = \RM^n$, the standard $\OO(n)$-representation induces a reducible
$\OO(n)$-repre\-sentation on $\Sym^k \T$. We denote by
$
\Sym^k_0 \T := \ker( \L : \Sym^k \T \rightarrow \Sym^{k-2} \T)
$
the space of trace-free symmetric $k$-tensors. 

It is well known that $\Sym^k_0 \T$ is an irreducible $\OO(n)$-representation and
we have the following decomposition into irreducible summands
$$
\Sym^k \T  \;\cong\;  \Sym^k_0 \T\;\oplus\; \Sym^{k-2}_0 \T \;\oplus\;  \ldots \ ,
$$
where the last summand
in the decomposition is $\RM$ for $k$ even and $\T$ for $k$ odd.
The summands 
$\Sym^{k- 2i}_0 \T$ are embedded into $\Sym^k \T$ via the map $\LL ^i$.  Corresponding to the
decomposition above any $K \in \Sym^k \T$ can be uniquely decomposed as
$$
K \;=\;  K_0 \;+\;  \LL K_1 \;+\;  \LL ^2K_2 \;+ \; \ldots
$$
with $K_i \in \Sym^{k-2i}_0 \T$, i.e. $\Lambda K_i = 0$. 
We will call this decomposition the {\it standard decomposition} of $K$.
In the following, the
subscript $0$ always  denotes the projection of an element from $\Sym^k \T$ onto its
component in $\Sym^k_0 \T$. Note that for any $v \in \T$ and $K \in \Sym^k_0 \T$ 
we have the following projection formula
\beq \label{projection}
(v \cdot K )_0 \;=\; v \cdot K \;-\; \tfrac{1}{n + 2 k-2}\, \LL  \,  (v \lrcorner \, K) .
\eeq
Indeed, using the commutator relations \eqref{commu} we have
$
\L ( \LL  \,  (v \lrcorner \,  K)) = (2n + 4(k-1)) \,  (v \lrcorner \,  K) 
$,
since $\Lambda $ commutes  with $v \lrcorner \, $ and $\Lambda K=0$. Moreover $\Lambda (v \cdot K) = 2\, v \lrcorner \, K$. Thus the right-hand side of 
\eqref{projection} is in the kernel of $\L$, i.e. it computes the projection $(v \cdot K)_0$.

\subsection{Conformal Killing tensors}\label{confKill}
Let $(M^n, g)$ be a Riemannian manifold with Levi-Civita connection $\nabla$. All the algebraic considerations above extend
to vector bundles over $M$, e.g. the $\OO(n)$-representation $\Sym^k \T$
defines the real vector bundle $\Sym^k \T M$. The $\OO(n)$-equivariant maps
$\LL $ and $\L$ define bundle maps between the corresponding bundles. The same
is true for the symmetric product $\cdot$ and the contraction $\lrcorner$. We
will use the same notation for the bundle maps on $M$.

Next we will define first order differential operators on sections of $\Sym^p \T M$. We have
$$
\dd : C^\infty(M,\Sym^k  \T M) \rightarrow C^\infty(M,\Sym^{k+1}\T M), \quad K \mapsto \sum_i \e_i \cdot \nabla_{\e_i}K ,
$$
where $\{\e_i\}$ denotes from now on a local orthonormal frame. The symmetric tensor $\dd K$ is the complete
symmetrisation of $\nabla K$, in the sense that 
\begin{equation}\label{dk}
\begin{split}
g(\dd K,X^{k+1})&=\sum_i  g (\nabla_{\e_i} K,\e_i\lrcorner X^{k+1}) = (k+1)\sum_i  g (\nabla_{\e_i} K, g(\e_i,X)X^{k})\\ &=(k+1) g (\nabla_{X} K, X^{k})
\end{split}
\end{equation}
for every $X\in \T M$.
The formal adjoint of 
$\dd$ is the divergence operator  $\dd^*$ defined by
$$
\dd^* : C^\infty(M,\Sym^{k+1} \T M) \rightarrow C^\infty(M,\Sym^{k}\T M), \quad K \mapsto -  \sum_i \e_i \lrcorner\, \nabla_{\e_i}K  .
$$
As an immediate consequence of  the definition we have that the operator $\dd$ acts as a \emph{derivation} on the algebra of symmetric tensors, i.e. for any $K_1 \in C^\infty(M,\Sym^k  \T M)$ and $K_2 \in  C^\infty(M,\Sym^l  \T M)$ the following equation holds
$$
\dd ( K_1 \cdot K_2) \;=\; \dd K_1 \cdot K_2 \,+\,K_1 \cdot \dd K_2.
$$
Moreover, an easy calculation proves that the operators $\dd$ and $\dd^*$ satisfy the  commutator relations:
\beq\label{commu3}
[ \, \Lambda, \,\dd^*\,]  \;=\; 0 \;=\;  [\,\LL ,\, \dd\, ], \quad [\, \Lambda, \,\dd\,] \; =\; -2 \dd^*, \quad [\, \LL ,\, \dd^*\, ] \;=\;  2 \,\dd .
\eeq

We also consider the operator 
$$
\dd_0: C^\infty(M,\Sym^k_0 \T M) \rightarrow C^\infty(M, \Sym^{k+1}_0\T M), \quad K \mapsto (\dd K)_0  .
$$
According to \eqref{projection}, we have $\dd_0 K=\dd K + \tfrac{1}{n + 2k-2}\, \LL  \,  \dd^* K$ for every $K\in C^\infty(M,\Sym^k_0 \T M)$. The formal adjoint $\dd_0^*: C^\infty(M,\Sym^{k+1}_0 \T M) \rightarrow C^\infty(M, \Sym^{k}_0\T M)$ is clearly equal to the restriction of $\dd^*$ to $C^\infty(M,\Sym^{k+1}_0 \T M)$.

A symmetric tensor $K \in C^\infty(M, \Sym^k \T M )$  is called {\it conformal Killing tensor} if there exists some symmetric tensor
$ k \in C^\infty(M, \Sym^{k-1} \T M)$  with $\dd K = \LL \, k$.  Note that $K$ is conformal Killing if and only if its trace-free part is 
conformal Killing. Indeed, since $\dd$  and $\LL$ commute, if
$K = \sum_{i\ge 0} \LL^i K_i$, with $K_i \in C^\infty(M,\Sym^{k-2i}_0 \T M)$ is the standard decomposition of $K$, then
$\dd K = \sum_{i \ge 0} \LL^i \dd K_i$, so $\dd K$ is in the image 
of $\LL$ if and only if $\dd K_0$ is in the image of $\LL$. More precisely we have the following characterisation (see also \cite[Lemma 3.3]{Heil-Moroianu-Semmelmann-16}): a symmetric tensor  $ K \in C^\infty(M, \Sym^k  \T M)$ is a conformal Killing tensor if and only if
\begin{equation}\label{confK}
\dd K_0 = - \tfrac{1}{n + 2k-2} \, \LL \dd^* K_0  .
\end{equation}
or, equivalently, if and only if the symmetric tensor $K$ satisfies the  condition $ \dd_0 K_0 = 0$. 

Let $E$ be a real  vector bundle over $M$ with connection $\nabla^{E}$. We extend $\dd$ and $\dd_0$ to twisted operators
\[
\dd : C^\infty(M,\Sym^k \T M \otimes E) \to C^\infty(M,\Sym^{k+1} \T M \otimes E),
\]
\[
\dd_0 : C^\infty(M,\Sym^k_0 \T M \otimes E) \to C^\infty(M,\Sym^{k+1}_0 \T M \otimes E),
\]
defined on decomposable elements  by 
$$
\dd (K \otimes \xi ) = \dd K \otimes \xi + \sum_i(\e_i \cdot K) \otimes \nabla_{\e_i}^{E}\xi ,\quad \dd_0 (K \otimes \xi ) =\dd_0 K \otimes \xi + \sum_i(\e_i \cdot K)_0 \otimes \nabla_{\e_i}^{E}\xi,
$$
obtained from the tensor product of Levi-Civita and $\nabla^{E}$ connections. In this case, sections in $\ker \dd $ are called \emph{twisted} Killing tensors and  sections in $\ker \dd_0$ are called
\emph{twisted} conformal Killing tensors.

%
\section{Weitzenb\"ock formulas}
%
Let $(M^n, g)$ be an oriented Riemannian manifold with Riemannian curvature tensor $R$. Let $\RO : \Lambda^2 \T M \rightarrow \Lambda^2 \T M$ be the curvature operator  defined
by $g(\RO (X \wedge Y), Z\wedge U) = R(X,Y,Z, U)$. 
With this convention we have
$\RO = - \, \Id$ on the standard sphere.

 Let 
$P=P_{\SO(n)}M$ be the frame bundle of $M$ and let $VM$ be the vector bundle associated to $P$ via a
$\SO(n)$-representation $\rho: \SO(n) \rightarrow \Aut (V)$, where $\Aut(V)$ denotes the isometries of a Euclidean vector space $(V, g_V)$. Then the curvature endomorphism $q(R) \in \End \, VM$ is  defined as
\begin{equation}\label{qr}
q(R) \;:=\; \tfrac12 \, \sum_{i, j} (\e_i \wedge \e_j)_\ast   \RO(\e_i \wedge \e_j)_\ast  .
\end{equation}
Here $\{\e_i\}, i = 1, \ldots n$, is a local orthonormal frame of $\T M$ and for $X\wedge Y \in \Lambda^2 \T M $
we define $(X \wedge Y)_\ast = \rho_\ast (X \wedge Y)$, where $\rho_\ast: \so(n) \rightarrow \End(V)$ is the differential
of $\rho$. In particular, 
the standard action of $\Lambda^2\T M$ on $\T M $ is written as
$
(X\wedge Y)_\ast\, Z \;=\; g( X,\,Z)\,Y \;-\; g(Y,\, Z) \, X = (Y \cdot X \,\lrcorner - X \cdot Y \, \lrcorner \,) Z
$.
This is compatible with
$$
g( (X \wedge Y)_\ast Z, U ) \;=\; g( X \wedge Y, Z \wedge U ) \;=\; g(X,Z) \, g(Y, U) \;-\; g( X, U)\, g( Y, Z)   .        
$$


Let $\T =\mathbb{R}^n$ be the 
standard representation of $\SO(n)$ defining the tangent bundle $\T M$. Then any $\SO(n)$-equivariant endomorphism $\p \in \End_{\SO(n)} (\T \otimes V)$ induces an $\SO(n)$-equivariant element $\tilde \p\in\Hom_{\SO(n)}(\T \otimes \T \otimes V, V)$ defined by
$$\tilde \p(a\otimes b \otimes v) := (a \,  \lrcorner\otimes \, \id) \, \p(b \otimes v),\qquad \forall\ a, b \in \T ,\ v \in V.$$

We note at this point that equivariant objects give rise to global parallel sections which we will denote by the same letter; for instance $p$ defines a parallel section $p \in C^\infty(M, \End(\T M \otimes V M))$. Important examples of such endomorphisms are the  orthogonal projections $\p_i,\, i =1, \ldots, N$, onto the summands in an $\SO(n)$-invariant decomposition $\T  \otimes V = V_1 \oplus \ldots \oplus V_N$.  Another example is the so-called
{\it conformal weight operator}  $B \in \End (\T \otimes V)$  introduced in \cite{Gauduchon-91}
(see also \cite{Calderbank-Gauduchon-Herzlich-00}) and defined as
$$
B(b \otimes v) := \sum_i \e_i \otimes (\e_i \wedge b)_\ast v.$$
The corresponding element $\tilde B \in \Hom (\T \otimes \T \otimes V, V)$ is given by
$$
\tilde B(a \otimes b \otimes v) = (a \wedge b)_* v  .
$$

For every equivariant orthogonal projector $\p \in \End_{\SO(n)} (\T \otimes V)$ we define a first order differential operator $P := \p \nabla $.

If $K$ is a section of $VM$, then $\nabla^2K = \sum_i \e_i \otimes \e_j \otimes \nabla^2_{\e_i, \e_j}K$ is a section of the bundle $\T M \otimes \T M \otimes VM$. Here for vector fields $X, Y$ on $M$ we denote $\nabla^2_{X, Y} K := \nabla_X \nabla_Y K - \nabla_{\nabla_X Y} K$; then the curvature endomorphism is given by $R_{X, Y} = \nabla^2_{X, Y} - \nabla^2_{Y, X}$. We can thus obtain natural second order operators by applying elements of the bundle $\Hom (\T M\otimes \T M\otimes V M, V M)$ to $\nabla^2 K$.

\begin{lemma}\label{bq}\cite[Proposition 3.1 and Lemma 3.6]{Uwe-06} The following relations hold:
$$
\tilde B  \nabla^2 = q(R), \qquad \tilde \p \nabla^2 =- P^*  P,
$$
where 
$P^*$ is the formal adjoint of $P$.
\end{lemma}
\begin{proof} Let $(\e_i)$ be a local orthonormal frame of $\T M$, parallel at the point where the computations are done (i.e. satisfying $\nabla_{\e_i} \e_j = 0$ for all $i, j$). 
The first formula is immediate:
$$\tilde B  \nabla^2=\sum_{i,j}(\e_i\wedge \e_j)_*\nabla^2_{\e_i,\e_j}=\tfrac12\sum_{i,j}(\e_i\wedge \e_j)_*R_{\e_i,\e_j}=q(R).$$

In order to prove the second one, we first compute the formal adjoint of $\nabla$. For all sections $\f$ of $VM$ and $\psi$ of $\T M\otimes VM$ we have 
\bea g(\nabla\f,\psi)&=&g\left(\sum_i \e_i\otimes\nabla_{\e_i}\f,\psi\right)=\sum_ig(\nabla_{\e_i}\f,(\e_i\lrcorner\otimes\id)\psi)\\
&=&\sum_i \e_i(g(\f,(\e_i\lrcorner\otimes\id)\psi))-\sum_ig(\f,(\e_i\lrcorner\otimes\id)\nabla_{\e_i}\psi).
\eea
Since the first term in the last equation is the codifferential of the 1-form $X\mapsto -g(\f,(X\lrcorner\otimes\id)\psi)$, we obtain $\nabla^*=-\sum_i(\e_i\lrcorner\otimes\id) \nabla_{\e_i}$. Using this formula, together with the fact that $\nabla \p=0$, $\p^2=\p$ and $\p^*=\p$, we then compute:
\bea \tilde \p \nabla^2&=&\tilde \p  \left(\sum_{i,j}\e_i\otimes \e_j\otimes\nabla^2_{\e_i,\e_j}\right)=\sum_{i,j}(\e_i\lrcorner\otimes \id)  \p  (\e_j\otimes\nabla^2_{\e_i,\e_j})\\
&=&\sum_{i,j}(\e_i\lrcorner\otimes \id) \nabla_{\e_i}  \left(\p (\e_j\otimes\nabla_{\e_j})\right) = \sum_{i}(\e_i\lrcorner\otimes \id) \nabla_{\e_i}  (\p  \nabla)\\
&=&-\nabla^*  \p \nabla=-\nabla^*  \p^*  \p \nabla=-P^*  P.
\eea
\end{proof}

Let us now consider the orthogonal projections $\p_s,\, s =1, \ldots, N$, onto the summands in an $\SO(n)$-invariant decomposition $\T  \otimes V = V_1 \oplus \ldots \oplus V_N$. The above result shows that whenever the conformal weight operator $B$ can be expressed as a linear combination of the projections $\p_s$, i.e. $B=\sum_s a_s\p_s$ for $a_s \in \mathbb{R}$, we obtain a corresponding Weitzenb\"ock formula:
\begin{equation}\label{we}
    q(R)=-\sum_s a_s \,P_s^*  P_s
\end{equation} 
on sections of $VM$, where $P_s$ are the first order differential operators defined by $P_s(K): = \p_s(\nabla K)$ for every section $K$ of $VM$, giving a section of $\T M \otimes V M$.

This universal Weitzenb\"ock formula was considered  for the first time in 
\cite{Gauduchon-91} and later extended and generalised for  other holonomy groups in \cite{Semmelmann-Weingart-10}. In fact,  the irreducible summands $V_s$ appearing in the decomposition  of $\mathrm T \otimes V$ are all pairwise non-isomorphic as $\SO(n)$ representations. Thus the projections $\p_s$ form a basis of $\End_{\SO(n)}(\mathrm T \otimes V)$ and there is an  explicit formula for expressing the coefficients $a_s$ in terms of the highest weights of $V$ and $V_s$ (see \cite[Corollary 3.4]{Semmelmann-Weingart-10}).


We consider now another $\SO(n)$-representation $E$ with an invariant scalar product and the corresponding vector bundle $EM$ over $M$, together with the induced metric. Let $\nabla^E$ be any metric connection on $E$, with curvature tensor denoted by $R^E$. For simplicity, we still denote by $\nabla^E$ the tensor product connection $\nabla\otimes\Id_{EM}+\Id_{VM}\otimes \nabla^E$ on $VM\otimes EM$. The projections $\p_s:\T  \otimes V\to \T  \otimes V$ define projections $\p_s\otimes \id:(\T  \otimes V)\otimes E\to (\T  \otimes V)\otimes E$ and, correspondingly, differential operators $P_s^E:=(\p_s\otimes \id) \nabla^E$, acting on $VM\otimes EM$.

Since $\sum_sa_s(\p_s\otimes\id)=B\otimes \id$ on $\T\otimes V\otimes E$, Lemma \ref{bq} yields at once 
\begin{equation}
\label{wei}\widetilde{B\otimes \id} (\nabla^E)^2=-\sum_s a_s  \, (P_s^E)^*  P_s^E,
\end{equation} 
acting on sections of $VM \otimes EM$. It remains to compute the action of the left-hand side operator. If $K\otimes \xi\in C^\infty(M,VM\otimes EM)$ is a decomposable section and $(\e_i)$ is an orthonormal frame parallel at the point of interest, we have
\begin{equation*}
\begin{split}
& (\widetilde{B\otimes \id} (\nabla^E)^2)(K\otimes \xi) \\
&=\widetilde{B\otimes \id}\left(\sum_{i,j}\e_i\otimes \e_j\otimes(\nabla^E)^2_{\e_i,\e_j}(K\otimes \xi)\right)\\
&=\widetilde{B\otimes \id}\left(\sum_{i,j}\e_i\otimes \e_j\otimes\left(\nabla^2_{\e_i,\e_j}K\otimes \xi+\nabla_{\e_i}K\otimes\nabla^E_{\e_j}\xi\right.\left.+\nabla_{\e_j}K\otimes\nabla^E_{\e_i}\xi+K\otimes(\nabla^E)^2_{\e_i,\e_j}\xi \right)\right)\\
&=\sum_{i,j}\left(((\e_i\wedge \e_j)_*\nabla^2_{\e_i,\e_j}K)\otimes \xi+((\e_i\wedge \e_j)_*\nabla_{\e_i}K)\otimes\nabla^E_{\e_j}\xi\right.
\\
&\left.\qquad\qquad\qquad\qquad\qquad+((\e_i\wedge \e_j)_*\nabla_{\e_j}K)\otimes\nabla^E_{\e_i}\xi+ (\e_i\wedge \e_j)_*K \otimes(\nabla^E)^2_{\e_i,\e_j}\xi \right)\\
&=(q(R)K)\otimes \xi+\tfrac12\sum_{i,j} (\e_i\wedge \e_j)_*K \otimes R^E_{\e_i,\e_j}\xi,
\end{split}
\end{equation*}
where the two middle terms cancel each other due to the skew-symmetry in $i$,$j$. Denoting by $q(R)^E$ the linear operator acting on (decomposable) sections of $VM\otimes EM$ by 
\begin{equation}\label{qrv}
    q(R)^E(K\otimes \xi):=(q(R)K)\otimes \xi+\tfrac12 \sum_{i,j}(\e_i\wedge \e_j)_*K\otimes R^E_{\e_i,\e_j}\xi,
\end{equation}
the previous relation \eqref{wei} implies the twisted Weitzenb\"ock-type formula 

\begin{equation}\label{twe}q(R)^E=-\sum_s a_s(P_s^E)^*  P_s^E\qquad \hbox{on } C^\infty(M,VM\otimes EM).
\end{equation} 

\medskip

We now consider the case of interest for us, namely $V=\Sym^k_0 \T$, where $\T:=\RM^n$ is the standard $\OO(n)$ representation of highest weight
$(1,0,\ldots , 0)$. 
Recall the classical decomposition into irreducible  $\OO(n)$ representations (e.g. see \cite{Semmelmann-Weingart-10}, p. 511-512):
\beq\label{deco}
\T \otimes \Sym^k_0 \T \;\cong\; \Sym^{k+1}_0\T \;\oplus\; \Sym^{k-1}_0\T \;\oplus\; \Sym^{k,1} \T ,
\eeq
where $\Sym^k_0 \T$  
is the irreducible representation of highest weight $(k, 0, \ldots, 0)$ and 
$\Sym^{k,1} \T$ is the irreducible representation of highest weight
$(k,1,0, \ldots, 0)$. We note that  $\Sym^{k+1}_0\T$ is the so-called Cartan summand.
Its highest weight is the sum of the highest weights of $\T$ and $\Sym^k_0\T$. 

For later use, let us first express the operator $q(R)$ on symmetric tensors in a more convenient way.
\begin{lemma}\label{qr1}
For every $K\in \Sym^k\T M$, the following relation holds:
$$q(R)(K)=-\sum_{i,j,k} R_{\e_i,\e_j}\e_k\lrcorner(\e_j\cdot \e_k\cdot (\e_i\lrcorner K)).$$
\end{lemma}
\begin{proof} For every skew-symmetric endomorphism $A$ of $\T M$ (identified with a section of $\Lambda^2 \T M$) we have $A_*K = \sum A\e_i\cdot (\e_i\lrcorner K).$ In particular, for $A=X\wedge Y$ we get $(X\wedge Y)_*K=Y\cdot (X\lrcorner K)-X\cdot(Y\lrcorner K)$.
We then compute using the symmetries of the Riemannian curvature tensor:
\begin{eqnarray*}q(R)(K)&=&\tfrac12\sum_{k,l}(\e_l\wedge \e_k)_*(R_{\e_l,\e_k})_*K=\tfrac12\sum_{i,k,l}(\e_l\wedge \e_k)_*(R_{\e_l,\e_k}\e_i\cdot (\e_i\lrcorner K))\\
&=&\sum_{i,k,l}\e_k\cdot \e_l\lrcorner(R_{\e_l,\e_k}\e_i\cdot \e_i\lrcorner K)=\sum_{i,k,l}\e_l\lrcorner(\e_k\cdot R_{\e_l,\e_k}\e_i\cdot (\e_i\lrcorner K))\\
&=&\sum_{i,j,k,l}\e_l\lrcorner(\e_k\cdot \e_j\cdot \e_i\lrcorner K)g(R_{\e_l,\e_k}\e_i,\e_j)=-\sum_{i,j,k}R_{\e_i,\e_j}\e_k\lrcorner(\e_k\cdot \e_j\cdot (\e_i\lrcorner K)).
\end{eqnarray*}
\end{proof}
Next we want  to describe projections and embeddings of the three summands. By \eqref{projection}, the map
$
\q_1 : \T \otimes \Sym^k_0\T \rightarrow \Sym^{k+1}_0 \T
$
onto the first summand is defined as
\beq\label{pi1}
\q_1(v \otimes K) \;:=\; (v \cdot K)_0 \;=\; v \cdot K \;-\; \tfrac{1}{n + 2k - 2}\, \LL \,  (v \lrcorner \,  K)  .
\eeq
The adjoint map
$
\q_1^* : \Sym^{k+1}_0 \T \rightarrow \T \otimes \Sym^k_0 \T
$
is easily computed to be 
\beq\label{pi10}
\q^*_1(K) = \sum_i \e_i \otimes (\e_i \lrcorner \, K).
\eeq
Note that for any vector $v\in \T$, the symmetric tensor
$v\lrcorner \, K$ is again trace-free, because $\, v \lrcorner \,$ commutes with $\L$.
Since $\q_1 \, \q^*_1 = (k+1) \, \Id$ on $\Sym^{k+1}_0\T$, we conclude that
\begin{equation}\label{pi11}
\p_1 \;:=\; \tfrac{1}{k+1} \, \q^*_1 \, \q_1 : \T \otimes \Sym^k_0 \T \;\rightarrow \;
\Sym^{k+1}_0 \T \;\subset \; \T \otimes \Sym^k_0 \T
\end{equation}
 is the orthogonal projection onto the irreducible summand of $\T \otimes \Sym^k_0 \T$ isomorphic to
 $\Sym^{k+1}_0\T $.

Similarly the map 
$
\q_2 : \T \otimes \Sym^k_0 \T \rightarrow \Sym^{k-1}_0 \T
$
onto the second summand in the decomposition \eqref{deco} is given by the contraction map
\beq\label{q10}
\q_2 ( v \otimes K)  := v \lrcorner \, K.
\eeq
In this case the adjoint map 
$
\q_2^*: \Sym^{k-1}_0 \T \rightarrow \T \otimes \Sym^k_0 \T
$
is computed to be
\beq\label{q11}
\q^*_2(K) \;=\; \sum_i \e_i \otimes (\e_i \cdot K)_0
\;=\; 
\sum_i \e_i \otimes \left(\e_i \cdot K - \tfrac{1}{n+2k-4} \, \LL  \, (\e_i \lrcorner \, K)\right). 
\eeq
It follows that
$$\;
\q_2 \, \q^*_2 = (n+k-1) \, \id \;-\; \tfrac{2k - 2}{n + 2k - 4} \, \id
=
\tfrac{(n+2k-2)(n+k-3)}{n+2k-4}\,\id ,
$$
so the projection onto the irreducible summand in
$\T \otimes \Sym^k_0 \T$ isomorphic to $\Sym^{k-1}_0 \T$  is given by
\begin{equation}\label{pi2}
\p_2 \;:=\; \tfrac{n+2k-4}{(n+2k-2)(n+k-3)} \, \q_2^* \, \q_2 
: \T\otimes \Sym^k_0 \T \;\rightarrow \; \Sym^{k-1}_0 \T \;\subset \; \T \otimes \Sym^k_0 \T ,
\end{equation}
valid for $n \geq 3$ and $k \geq 1$. The projection $\p_3$ onto the third irreducible summand in  $\T \otimes \Sym^k_0 \T$ is obviously given by $\p_3 = \id - \p_1 - \p_2$ .

The algebraic considerations above extend to vector bundles over $M$. In particular, the operators $\dd_0: C^\infty(M, \Sym^k_0 \T M)\to  C^\infty(M, \Sym^{k+1}_0 \T M)$ and $\dd_0^*: C^\infty(M, \Sym^k_0 \T M)\to  C^\infty(M, \Sym^{k-1}_0 \T M)$
introduced above can be described as
\begin{equation}\label{dd0}
    \dd_0 K  = \q_1\nabla K,\qquad\dd_0^* K = - \q_2\nabla K ,
\end{equation}
for every section $K \in C^\infty(M, \Sym^k_0 \T M)$. 
By \eqref{confK} together with  \eqref{pi1} and \eqref{pi11} we see that the kernel of 
$P_1= \p_1   \nabla$ consists exactly of trace-free  conformal Killing tensors. The kernel  of $P_2= \p_2   \nabla$ are the divergence 
free tensors, i.e. tensors in $\ker \dd_0^*$.

An easy calculation using the explicit formulas for $\q_1$ and $\q_2$ proves the following relation on $\T \otimes \Sym^k_0\T$ (see \cite[Proposition 6.1]{Heil-Moroianu-Semmelmann-16}):
$$
B = k\, \p_1 \; - \; (n+k-2) \, \p_2 \;-\, \p_3  .
$$
As explained above, this yields the Weitzenböck-type formula.
\begin{equation}\label{wei2}
 q(R) K = - k \, P_1^*P_1  K  \,+\, (n+k-2)\, P_2^*P_2  K  \,+ \, P_3^*P_3\, K.
\end{equation}
for any section $K$ of $\Sym^k_0 \,\T M$. In the present situation
it is easy to get the coefficients for $B$ by a direct calculation.
Alternatively one can use the general formula in terms of highest
weights mentioned above.


Now, if $EM$ is a Euclidean vector bundle associated to a representation $E$ of $\SO(n)$ with metric connection $\nabla^E$, we denote by $\p_i^E:=\p_i\otimes\Id_{E}$, by  $\q_i^E:=\q_i\otimes\Id_{E}$ and by
\begin{equation}\label{pi}
P^E_i:=\p_i^E \nabla^E,\quad i = 1, 2, 3,
\end{equation}
and obtain as before the twisted counterpart of \eqref{wei2}
\begin{equation}\label{twistw}
 q(R)^E  = -k \, (P_1^E)^*P_1^E \,+ \,(n+k-2)  \, (P_2^E)^*P_2^E \,+ \,(P_3^E)^*P_3^E,
\end{equation}
acting on sections of $\Sym^k_0 \,\T M\otimes EM$.

Since $\p_i^E$ are orthogonal projectors, we have $(\p_i^E)^* \p_i^E=\p_i^E$, so using \eqref{pi11} and recalling that $\dd_0 = \q_1^E \nabla^E$ (similarly to \eqref{dd0}), we obtain
\[
(P_1^E)^*P_1^E=(\nabla^E)^* (\p_1^E)^* (\p_1^E) \nabla^E=(\nabla^E)^* (\p_1^E) \nabla^E=\tfrac1{k+1}(\nabla^E)^* (\q_1^E)^* \q_1^E \nabla^E=\tfrac1{k+1}\dd_0^* \dd_0,
\]
and similarly using \eqref{pi2}, yields
 $$(P_2^E)^*P_2^E= \tfrac{n+2k-4}{(n+2k-2)(n+k-3)}\dd_0 \dd_0^*.$$
From these last two equations, together with \eqref{twistw} we obtain the following

 \begin{pro}[Twisted Weitzenböck formula]\label{tw}
 The following formula holds for sections of $\Sym^k_0 \,\T M\otimes EM$:
 \begin{equation}\label{twistw1}
 q(R)^E  = -\tfrac{k}{k+1} \dd_0^* \dd_0 +  \tfrac{(n+k-2)(n+2k-4)}{(n+2k-2)(n+k-3)} \dd_0 \dd_0^* + (P_3^E)^*P_3^E.
\end{equation}
 \end{pro}

\section{Fourier analysis in the fibers of the unit tangent bundle}

\label{ssection:fourier}

Further details on this section can be found in \cite{Paternain-99}, \cite[Section 2]{Paternain-Salo-Uhlmann-15}.

\subsection{Functions on the unit tangent bundle}

 We denote by $SM$ the unit tangent bundle of $(M,g)$ and by $\pi : SM \rightarrow M$ the projection on the base. There is a canonical splitting of the tangent bundle to $SM$ as:
\[
T(SM) =  \V \oplus \HH \oplus \mathbb{R} X,
\] 
where $X$ is the geodesic vector field, $\V := \ker d \pi$ is the vertical space and $\HH$ is the horizontal space defined in the following way. Define the \emph{connection map} $\mc{K} : T(SM) \rightarrow \T M$ as follows: let $v \in SM, w \in T_{v}(SM)$ and a curve $(-\eps,\eps) \ni t \mapsto v(t) \in SM$ such that $v(0)=v, \dot{v}(0)=w$. Denoting $x(t):=\pi(v(t))$, we have $\mc{K}_{v}(w) := \nabla_{\dot{x}(t)} v(t)|_{t=0}$. We denote by $g_{\mathrm{Sas}}$ the Sasaki metric on $SM$, which is the canonical metric on the unit tangent bundle, defined by:
\[
g_{\mathrm{Sas}}(w,w') := g(d \pi(w), d\pi(w')) + g(\mc{K}(w),\mc{K}(w')).
\]
Then the horizontal bundle $\HH$ is defined as the orthogonal complement of $X$ inside $\ker \mc{K}$. 

We define the \emph{normal bundle} $\mc{N} \rightarrow SM$ whose fiber at $v \in SM$ is given by $\mc{N}_v := v^\bot \subset T_{\pi(v)}M$. Then $d \pi : \HH \rightarrow \mc{N}, \mc{K} : \V  \rightarrow \mc{N}$ are both isometries and all these bundles over $SM$ are isomorphic. We will freely identify them in the following. In particular, we will think of the normal bundle $\mc{N}$ as the tangent bundle to the spheres.

For $x \in M$, the unit sphere
\[
S_xM = \left\{ v \in T_xM ~|~ |v|^2_x = 1\right\} \subset SM
\]
(endowed with the Sasaki metric) is isometric to the canonical sphere $(\Ss^{n-1},g_{\mathrm{can}})$. We denote its Laplace operator by $\Delta_x$.  Let $\Delta_{\V}$ be the vertical Laplacian acting on $f \in C^\infty(SM)$ as $\Delta_{\V}f(v) := \Delta_{\pi(v)}(f|_{S_{\pi(v)}M})(v)$, for every $v\in SM$. For $k \geq 0$ and $x\in M$, we introduce
\[
\Omega_k(x) = \ker\left(\Delta_x- k(n+k-2)\Id\right),
\]
the spherical harmonics of degree $k$. Observe that $\Omega_k \rightarrow M$ defines a vector bundle over $M$, and that $C^\infty(M, \Omega_k)$ is naturally identified with a subspace of $C^\infty(SM)$. Given $f \in C^\infty(SM)$, it can be decomposed as $f = \sum_{k \geq 0} {f}_k$ where ${f}_k \in C^\infty(M,\Omega_k)$ is the projection of $f$ onto spherical harmonics of degree $k$. We call \emph{Fourier degree} of $f$, denoted by $\mathrm{deg}(f)$, the maximal integer $k_0 \in \ZM_{\geq 0}$ (if it exists) such that ${f}_{k_0} \neq 0$; otherwise we set $\deg(f) = \infty$. We will also say that $f$ has \emph{finite Fourier content} if its degree is finite, that it is \emph{odd} (resp. \emph{even}) if it only contains odd (resp. even) spherical harmonics. 

It can be proved that the operator $X$ has the following mapping properties (see \cite[Section 3]{Paternain-Salo-Uhlmann-15}):
\[
X : C^\infty(M,\Omega_k) \rightarrow C^\infty(M,\Omega_{k+1}) \oplus C^\infty(M,\Omega_{k-1}).
\]
This is understood in the following sense: a section ${f}_k \in C^\infty(M,\Omega_k)$ defines in particular a smooth function in $C^\infty(SM)$ which we can differentiate in the $X$-direction and this only contains spherical harmonics of degree $k-1$ and $k+1$. Taking the projection on higher degree (resp. lower degree), we obtain an operator $X_+ : C^\infty(M,\Omega_k) \rightarrow C^\infty(M,\Omega_{k+1})$ of gradient type i.e. with injective principal symbol (resp. $X_- : C^\infty(M,\Omega_k) \rightarrow C^\infty(M,\Omega_{k-1})$ of divergence type) such that $X=X_+ + X_-$ and $X_+^*=-X_-$ (the latter being a mere consequence of the fact that $X^*=-X$ as $X$ preserves the Sasaki volume (also known as the Liouville measure) on $SM$). As $X_+$ acting on spherical harmonics of degree $k$ has injective principal symbol, its kernel is finite dimensional by elliptic theory. 
As a consequence of Lemma \ref{lemma:link-d} we will later see that elements in the kernel of $X_+$
correspond to conformal Killing tensors, i.e. elements in the kernel of $\dd_0$ as defined in Section \ref{confKill}.


\subsection{Twist by a vector bundle}

\label{sssection:twisted-discussion}

Let ${E} \rightarrow M$ be a real vector bundle over $M$ equipped with a metric connection $\nabla^{E}$. Consider the pullback bundle $\mc{E}:=\pi^*E \rightarrow SM$ equipped with the pullback connection $\nabla^{\mc{E}}:=\pi^*\nabla^E$ and introduce the first order differential operator
\[
\X := \nabla^{\mc{E}}_X : C^\infty(SM,\mc{E})  \rightarrow C^\infty(SM,\mc{E}).
\]

The connection $\nabla^{\mc{E}}$ also gives rise to differential operators:
\[
\nabla^{\mc{E}}_{\HH},\ \nabla^{\mc{E}}_{\V} : C^\infty(SM,\mc{E}) \rightarrow C^\infty(SM,\mc{N} \otimes \mc{E} ),
\]
defined in the following way: for every section $f \in C^\infty(SM,\mc{E})$, the covariant derivative $\nabla^{\mc{E}} f \in C^\infty(SM, T^*(SM) \otimes \mc{E})$ can be identified with an element of $C^\infty(SM,T(SM) \otimes \mc{E})$ by applying the musical isomorphism $T^*(SM) \rightarrow T(SM)$ induced by the Sasaki metric. Using the orthogonal projections $\bullet_{\HH}$ and $\bullet_{\V}$ of $T(SM)$ onto $\HH$ and $\V$, respectively, one can then define the operators: 
\[
\nabla^{\mc{E}}_{\HH} f := d\pi( (\nabla^{\mc{E}} f)_\HH), \qquad \nabla^{\mc{E}}_{\V}f := \mc{K} ((\nabla^{\mc{E}} f)_\V),
\]
which take values in the bundle $\mc{N} \otimes \mc{E} \rightarrow SM$. In local coordinates, these operators have explicit expressions in terms of the connection $1$-form and we refer to \cite[Lemma 3.2]{Guillarmou-Paternain-Salo-Uhlmann-16} for further details.

If $(\xi_1,\ldots, \xi_r)$ is a local orthonormal frame of $E$, then smooth local sections $f$ of $\mc{E}$ can be written as:
\[
f(v) = \sum_{j=1}^r f^{(j)}(v) \xi_j(x) \in \mc{E}_x,\qquad\forall v\in S_xM,
\]
where $f^{(j)} \in C^\infty(SM)$ are locally defined functions. As before, each $f^{(j)}$ can be in turn decomposed into spherical harmonics. In other words, we can write $f = \sum_{k \geq 0}{f}_k$, where ${f}_k \in C^\infty(M, \Omega_k \otimes E)$.

As before, we can define the degree of $f \in C^\infty(SM,\mc{E})$ and we say that $f$ has \emph{finite Fourier content} if its expansion in spherical harmonics only contains a finite number of terms. The operator $\X$ maps
\begin{equation}
\label{eq:XX}
\X : C^\infty(M, \Omega_k \otimes E) \rightarrow C^\infty(M, \Omega_{k-1} \otimes E) \oplus C^\infty(M, \Omega_{k+1} \otimes E)
\end{equation}
and can be decomposed as $\X = \X_+ + \X_-$, where, if $u \in C^\infty(M,\Omega_k \otimes E)$, $\X_\pm u \in C^\infty(M,\Omega_{k \pm 1} \otimes E)$ denote the orthogonal projections on the twisted spherical harmonics of degree $k \pm 1$. The operator $\X_+$ is elliptic and thus has finite-dimensional kernel whereas $\X_-$ is of divergence type. Moreover, $\X_+^* = -\X_-$, where the adjoint is computed with respect to the canonical $L^2$ scalar product on $SM$ induced by the Sasaki metric and the metric on $E$. We also refer to the original articles of Guillemin-Kazhdan \cite{Guillemin-Kazhdan-80, Guillemin-Kazhdan-80-2} for a description of these facts and to \cite{Guillarmou-Paternain-Salo-Uhlmann-16} for a more modern exposition. It was shown in \cite[Theorem 4.1]{Guillarmou-Paternain-Salo-Uhlmann-16} (see also \cite[Corollary 4.2]{Cekic-Lefeuvre-Moroianu-Semmelmann-22} for a short argument) that flow-invariant sections, i.e. smooth sections in $\ker \X$ have \emph{finite Fourier content}.

\section{Symmetric tensors versus polynomial functions}

Considering symmetric tensors in $\Sym^k \T M$ as (pointwise) homogeneous polynomials of degree $k$ on $\T M$, gives linear maps
\begin{equation}
\label{equation:pullback}
\pi^*_k : C^\infty(M,\Sym^k \T M) \to  C^\infty(SM), \hspace{1cm} (\pi^*_k K)(v) := \tfrac1{k!}g(K,v^k).
\end{equation}
Note here that $\tfrac1{k!} v^k = v \otimes \dotsm \otimes v$, where the tensor product is repeated $k$ times.
\begin{lemma}
The linear map 
$$\pi^*:=\bigoplus_{k\ge 0}\pi_k^*:C^\infty\big(M,\Sym \T M\big)\to C^\infty(SM)$$
is an algebra homomorphism.
\end{lemma}
\begin{proof}
Using the bilinearity of the symmetric product it suffices to prove $\pi^*(a \cdot b) = (\pi^*a)(\pi^*b)$ where $a = a_1 \dotsm a_k$ and $b = b_1 \dotsm b_l$, for some $a_i, b_j \in C^\infty(M, \Sym^1 \T M)$. But this follows from
\begin{align*}
    (\pi^*a)(\pi^*b) &= \tfrac1{k!}g(a, v^k) \tfrac1{l!} g(b, v^l) = g(a_1, v) \dotsm g(a_k, v) g(b_1, v) \dotsm g(b_l, v)\\ 
    &= \tfrac1{(k + l)!} g(a \cdot b, v^{k + l}) = \pi^*(a \cdot b),
\end{align*}
which completes the proof.
\end{proof}

The following is standard and is a consequence of the identification of spherical harmonics with harmonic homogeneous polynomials (e.g. see \cite{Berger-Gauduchon-Mazet-71}, Chapter C.I).

\begin{lemma}
The above maps induce pointwise isomorphisms
\begin{equation}
\label{equation:pullback2}
\pi^*_k : \Sym^k_0 \T_x M \xrightarrow{\sim}  \Omega_{k}(x), 
\end{equation}
for every $x\in M$ and for every integer $k\ge 0$. 
\end{lemma}

If $E$ is any vector bundle over $M$ and $\mc{E}$ is its pull-back to $SM$, the spaces of sections $C^\infty(M,\Sym \T M\otimes E)$ and $C^\infty(SM,\mc{E})$ are modules over the algebras $C^\infty(M,\Sym \T M)$ and $C^\infty(SM)$ respectively, and 
we can extend the linear maps above to linear maps
\begin{equation}
\label{equation:pullbackE}
\pi^*_k : C^\infty(M,\Sym^k \T M\otimes E) \to  C^\infty(SM,\mc{E}), \qquad  \pi^*_k (K\otimes\xi)(v) :=\pi_k^*(K)\pi^*\xi
\end{equation}
compatible with the module structures in sense that 
\begin{equation}\label{product}\pi_k^*(K)\cdot\pi_l^*(K'\otimes \xi)=\pi_{k+l}^*((K\cdot K')\otimes \xi)
\end{equation}
 for every $K\in C^\infty(M,\Sym^k \T M)$, $K'\in C^\infty(M,\Sym^l \T M)$ and $\xi\in C^\infty(M,E)$. In particular, since 
 $$\pi_2^*(\LL)(v)=\tfrac12g(\LL,v\cdot v)=\tfrac12g(v\lrcorner\LL, v)=\tfrac12g(2v, v)=1,\qquad\forall v\in SM,$$
 we have $\pi_{k+2}^*(\LL K)=\pi_k^*(K)$ for every  $K\in C^\infty(M,\Sym^k \T M)$.

We now relate the operators $\X$, $\X_+$ and $\X_-$ with the operators
$\dd$, $\dd_0$ and $\dd_0^*$ defined in Section \ref{confKill}. 

\begin{lemma}
\label{lemma:link-d}
The following relation holds on sections of $\Sym^k\T M \otimes E$:
\begin{equation}
\label{27} \X \pi_k^* =\pi_{k+1}^* \dd,
\end{equation}
while on sections of $\Sym^k_0\T M \otimes E$ we have:
\begin{eqnarray}
\label{28} \X_+ \pi_k^*=&\pi_{k+1}^*  \dd_0,\\
\label{29} \X_- \pi_k^*=&-\tfrac1{n+2k-2}\pi_{k-1}^*  \dd_0^*.
\end{eqnarray}
\end{lemma}

\begin{proof} For the first equation, it is enough to check it on decomposable sections $\Psi=K\otimes \xi$, with $K\in  C^\infty(M,\Sym^k \T M)$ and $\xi\in  C^\infty(M,E)$. Then 
$$\X\pi_k^*\Psi=\nabla^{\mc{E}}_X(\pi_k^*(K)\pi^*\xi)=X(\pi_k^*(K))\pi^*\xi+\pi_k^*(K)\pi^*(\nabla^{E}_X\xi)$$
and 
\begin{eqnarray*}\pi_{k+1}^* \dd \Psi&=&\pi_{k+1}^*(\dd K\otimes\xi+\e_i\cdot K\otimes\nabla^E_{\e_i}\xi)=\pi_{k+1}^*(\dd K)\pi^*\xi+g(\e_i,v)\pi_k^*(K)\nabla^{\mc{E}}_{\e_i}\pi^*\xi\\
&=&\pi_{k+1}^*(\dd K)\pi^*\xi+\pi_k^*(K)\nabla^{\mc{E}}_{X}\pi^*\xi=\pi_{k+1}^*(\dd K)\pi^*\xi+\pi_k^*(K)\pi^*(\nabla^E_{v}\xi),
\end{eqnarray*}
where we identified $\e_i$ with their horizontal lifts to $SM$ and used that $d\pi(X) = v$. It remains to prove that $X(\pi_k^*(K))=\pi_{k+1}^*(\dd K)$. Let $v\in SM$ be any vector and denote by $x:=\pi(v)$. The geodesic in $M$ determined by $(x,v)$ will be denoted by $\gamma_t$. Then the integral curve of $X$ through $v$ is $\dot\gamma_t$. We can thus compute
\begin{eqnarray*}
X(\pi_k^*(K))(v)&=&\frac{d}{dt}\bigg|_{t=0}\pi_k^*(K)(\dot\gamma_t)=\frac1{k!}\frac{d}{dt}\bigg|_{t=0}g(K,\dot\gamma_t^k)\\&=&\frac1{k!}g(\nabla_{\dot\gamma_0}K,\dot\gamma_0^k)
\stackrel{\eqref{dk}}{=}\frac1{(k+1)!}g(\dd K,v^{k+1})=\pi_{k+1}^*(\dd K)(v),
\end{eqnarray*}
where in the third equality we used that $\nabla_{\dot\gamma_0} \dot \gamma_0 = 0$. This proves \eqref{27}. Using this equation applied to some twisted trace-free symmetric tensor $\Psi\in C^\infty(M, \Sym^k_0\T M \otimes E)$ together with \eqref{projection} we then obtain
\begin{eqnarray*}\X_+ \pi_k^*\Psi+\X_- \pi_k^*\Psi&=&\pi_{k+1}^*\dd \Psi=\pi_{k+1}^*\left( \dd_0 (\Psi)-\tfrac{1}{n+2k-2}\LL\dd_0^*(\Psi)\right)\\&=&\pi_{k+1}^*( \dd_0 (\Psi))-\tfrac{1}{n+2k-2}\pi_{k-1}^*( \dd_0^* (\Psi)).
\end{eqnarray*}
Comparing the components in $\Omega_{k+1}\otimes E$ and $\Omega_{k-1}\otimes E$ yields \eqref{28}--\eqref{29} at once.
\end{proof}

Consider now the operator $\nabla_\V:C^\infty(SM,\mc{E})\to C^\infty(SM,\mc{N}\otimes\mc{E})\subset C^\infty(SM,\pi^*(\T M)\otimes\mc{E})$ and its formal adjoint $\nabla_\V^*:C^\infty(SM,\pi^*(\T M)\otimes\mc{E})\to C^\infty(SM,\mc{E})$. Define the bundle map
$$
    S_k:\Sym^k\T M\otimes E\to \Sym^{k-1}\T M\otimes(E\otimes \T M), \quad S_k(K\otimes \xi) := \sum_i(\e_i\lrcorner K)\otimes(\xi\otimes\e_i),
$$
where $(\e_i)$ is some local orthonormal frame of $\T M$. Let $\pi_{\mc{N}}:\pi^*\T M\to\mc{N}$ be the orthogonal projection. By definition, for every section $K\otimes \xi$ of $\Sym^k\T M\otimes E$ and at any $v\in SM$ we have:
\begin{eqnarray*}
\pi_{k-1}^*S_k(K\otimes \xi)&=&\pi_{\mc{N}}\pi_{k-1}^*S_k(K\otimes \xi)+\sum_i\frac1{(k-1)!}g(\e_i\lrcorner K,v^{k-1})\,(g(\e_i,v)v\otimes\xi)\\
&=&\pi_{\mc{N}}\pi_{k-1}^*S_k(K\otimes \xi)+k\pi_{k}^*(K\otimes\xi)\otimes v,
\end{eqnarray*}
thus showing that for every $\Psi\in  C^\infty(M,\Sym^k\T M\otimes E)$, 
\begin{equation}
\label{kp} \pi_{k-1}^*S_k\Psi=\pi_{\mc{N}}\pi_{k-1}^*S_k\Psi+k\,\pi_{k}^*\Psi\otimes v.
\end{equation}

It is possible to give a formula relating $S_k$ and $\nabla_{\V}$:

\begin{lemma}
\label{lemma:link-v}
The following relation holds for sections of $\Sym^k_0 \T M\otimes E$:
\begin{equation}
\label{30} \nabla_\V \pi_k^*=\pi_{\mc{N}}\pi_{k-1}^* S_k.
\end{equation}
Moreover, for every $K\otimes\xi\in  C^\infty(M,\Sym^k_0 \T M\otimes E)$, and $w\in C^\infty(M,\T M)$,
\begin{equation}
\label{31} \nabla_\V^* \pi_k^* (K\otimes (w\otimes \xi))=-\pi_{k-1}^* ((w\lrcorner K)\otimes \xi)+k\pi_{k+1}^*((w\cdot K)\otimes \xi).
\end{equation}
\end{lemma}
\begin{proof} Let $v, w\in S_xM$ with $w \perp v$. We denote by $v_t:=\cos t\,v+\sin t\, w$ the curve in $S_xM$ which satisfies $v_0=v$ and $\dot v_0=w$. We then compute 
\begin{equation}\label{deriv}
\begin{split}
    w(\pi_k^*(K))&=\tfrac{d}{dt}\big|_{t=0} \pi_k^*(K)(v_t) = \tfrac1{k!}\tfrac{d}{dt}\big|_{t=0}g(K,v_t^k) = \tfrac1{(k-1)!}g(K,w\cdot v^{k-1})\\
    &=\tfrac1{(k-1)!}g(w\lrcorner K,v^{k-1}),
\end{split}
\end{equation}
whence for $\Psi := K \otimes \xi$ we have
$$ \nabla_\V \pi_k^*(\Psi)(w)=\nabla^{\mc{E}}_w(\pi_k^*(K)\pi^*\xi)=w(\pi_k^*(K))\pi^*\xi=\tfrac1{(k-1)!}g(w\lrcorner K,v^{k-1})\pi^*\xi,
$$
where we identified $w$ with its vertical lift. Then, computing the right hand side at the point $v$ yields
\begin{eqnarray*}\pi_{\mc{N}}\pi_{k-1}^* S_k (\Psi)(w)&=&\pi_{k-1}^*  \left(\sum_i(\e_i\lrcorner K)\otimes(\e_i\otimes\xi)\right)(w) = \sum_i \pi_{k-1}^*(\e_i\lrcorner K) g(\e_i,w)\pi^*\xi\\
&=&\tfrac1{(k-1)!}\sum_i g(\e_i\lrcorner K, v^{k-1}) g(\e_i,w)\pi^*\xi=\tfrac1{(k-1)!}g(w\lrcorner K, v^{k-1})\pi^*\xi,
\end{eqnarray*}
thus proving \eqref{30}.

We now remark that since $SM\to M$ is a Riemannian submersion, the formal adjoint of the operator $\nabla_\V$ can be written as
$$
    \nabla_\V^*(\sigma\otimes\psi)=-\sum_i \mathbf{f}_i\lrcorner \nabla^{\mc E}_{\mathbf{f}_i}(\sigma\otimes\psi)
$$
for all sections $\sigma\in C^\infty(SM,\pi^*\T M)$, and $\psi\in C^\infty(SM,\mc{E})$, where $(\mathbf{f}_i)$ denotes a local orthonormal frame of $\V\subset \T(SM)$ and the interior product is taken with respect to the bilinear form $\V\otimes \pi^*\T M\to \RM$ determined by the metric $g$, after identification of $\V_v$ with the orthogonal complement of $v$ in $\pi^*(\T M)_v$ for every $v\in SM$. We then denote by $w^\bot := w-g(w,v)v\in\V_v$ at some $v\in SM$ and compute:
\begin{equation*}
\begin{split}\nabla_\V^* \pi_k^* (K\otimes (w\otimes \xi))&=-\sum_i \mathbf{f}_i\lrcorner \nabla^{\mc E}_{\mathbf{f}_i}\big(\pi_k^* (K\otimes (w\otimes \xi))\big)=-\sum_i \mathbf{f}_i\lrcorner\big( \mathbf{f}_i(\pi_k^* (K))\pi^*(w\otimes \xi)\big)\\
&=-w^\bot(\pi_k^* (K))\pi^*(\xi)\stackrel{\eqref{deriv}}{=}-\tfrac1{(k-1)!}g(w^\bot\lrcorner K,v^{k-1})\pi^*(\xi)\\
&=-\pi_{k-1}^*(w\lrcorner K)\pi^*(\xi)+\tfrac1{(k-1)!}g(K,v^{k})\pi_1^*(w)\pi^*(\xi)\\
&=-\pi_{k-1}^* ((w\lrcorner K)\otimes \xi)+k\pi_{k+1}^*((w\cdot K)\otimes \xi) .
\end{split}
\end{equation*}
\end{proof}

Finally, we compute the action of the operator $P_3^{E}$ pulled back to the unit sphere bundle.

\begin{lemma}
For every $\Psi\in  C^\infty(M,\Sym^k_0 \T M\otimes E)$, and $w\in C^\infty(M,\T M)$, 
\begin{equation}
\label{equation:z}
Z_k \pi_k^* \Psi=\pi_k^*P^E_3 \Psi,\qquad Z_k^*  \pi_k^*(w\otimes\Psi)=\pi_k^*((P^E_3)^*(w\otimes\Psi))
\end{equation}
where  $Z_k:C^\infty(M,\Omega_k \otimes E)\to C^\infty(SM,\mc{N}\otimes \mc{E})\subset C^\infty(SM,\pi^* \T M \otimes\mc{E})$ is the operator defined by
\begin{equation}\label{eq:Z}
    Z_k f:=\nabla_\HH f-\tfrac1{k+1}\nabla_\V\X_+f+\tfrac{1}{n+k-3}\nabla_\V\X_-f
\end{equation}
and
$$P_3^E:C^\infty(M,\Sym^k_0\T M\otimes E)\to C^\infty(M,\Sym^k_0\T M\otimes(\T M\otimes E))$$ is the first order differential operator appearing in \eqref{pi}.
\end{lemma}
\begin{proof}
   It is enough to check the first relation, the second following by taking the metric adjoints. By definition we have $P^E_3=\nabla^E-P_1^E-P_2^E$.   
   Let us first explicit the last two operators. Using \eqref{pi1}--\eqref{pi11} we compute
   $$P_1^E\Psi=\frac1{k+1}(\q_1^E)^*\q_1^E(\nabla^E\Psi)=\frac1{k+1}\sum_i(\e_i\lrcorner \dd_0\Psi)\otimes\e_i=\frac1{k+1}S_{k+1}\dd_0\Psi.$$
   From \eqref{28}, \eqref{kp} and \eqref{30} we thus get at any $v\in SM$:
   \begin{equation}\label{p1}
       \pi_k^*P_1^E\Psi=\frac1{k+1}\pi_{\mc{N}}\pi_k^*S_{k+1}\dd_0\Psi+\pi_{k+1}^*\dd_0\Psi\otimes v=\frac1{k+1}\nabla_\V\X_+\pi_k^*\Psi+\X_+\pi_k^*\Psi\otimes v.
   \end{equation}
   Similarly, from \eqref{q10}--\eqref{pi2} we obtain
   \begin{eqnarray*}
       P_2^E\Psi&=&\tfrac{n+2k-4}{(n+2k-2)(n+k-3)} (\q_2^E)^* \q_2^E(\nabla^E\Psi)=-\tfrac{n+2k-4}{(n+2k-2)(n+k-3)} (\q_2^E)^* \dd_0^*\Psi\\
       &=&-\tfrac{n+2k-4}{(n+2k-2)(n+k-3)} \sum_i\left((\e_i\cdot \dd_0^*\Psi)\otimes\e_i-\tfrac{1}{n+2k-4} \LL(\e_i\lrcorner \dd_0^*\Psi)\otimes\e_i\right).
        \end{eqnarray*}
       Applying this equation at some $v\in SM$ and using \eqref{29}, \eqref{kp} and \eqref{30} we get:
        \begin{equation}
            \begin{split}
         \pi_k^*P_2^E\Psi&=-\tfrac{n+2k-4}{(n+2k-2)(n+k-3)} \sum_i\left((\pi_1^*\e_i\cdot \pi_{k-1}^*\dd_0^*\Psi)\otimes\e_i-\tfrac{1}{n+2k-4}\pi_{k-2}^*(\e_i\lrcorner \dd_0^*\Psi\otimes\e_i)\right)  \\
          &= \tfrac{n+2k-4}{n+k-3} \sum_i(g(\e_i,v)\cdot \X_-\pi_{k}^*\Psi)\otimes\e_i + \tfrac{1}{(n+2k-2)(n+k-3)}(\pi_{k-2}^*S_{k-1}\dd_0^*\Psi) \\
          &=\tfrac{n+2k-4}{n+k-3}\X_-\pi_{k}^*\Psi\otimes v
          +\tfrac{1}{(n+2k-2)(n+k-3)}(\nabla_\V\pi_{k-1}^*\dd_0^*\Psi+(k-1)\pi_{k-1}^*\dd_0^*\Psi\otimes v)\\
          &=\X_-\pi_{k}^*\Psi\otimes v-\tfrac{1}{n+k-3}\nabla_\V\X_-\pi_{k}^*\Psi.
            \end{split}
        \end{equation}
        Finally, using the fact that $\pi:SM\to M$ is a Riemannian submersion, we readily obtain at any $v\in SM$:
        \begin{equation}\label{40}
            \pi_k^*(\nabla^E\Psi)=\nabla_\HH\pi_k^*\Psi+\X\pi_k^*\Psi\otimes v.
        \end{equation}
   From \eqref{p1}--\eqref{40} we thus get:
    \begin{equation*}
            \pi_k^*P^E_3 \Psi=\pi_k^*(\nabla^E\Psi-P_1^E\Psi-P_2^E\Psi)=\nabla_\HH\pi_k^*\Psi-\frac1{k+1}\nabla_\V\X_+\pi_k^*\Psi+\tfrac{1}{n+k-3}\nabla_\V\X_-\pi_{k}^*\Psi,
    \end{equation*}
    which proves the lemma.
\end{proof}
We note that as a consequence of the preceding lemma, the operator $Z_k$ defined in \eqref{eq:Z} does not change the degree of the section it acts on (since $P_3^E$ does not change the degree).

\section{Twisted Pestov identity}

The Pestov identity is a classical identity in Riemannian geometry, see \cite{Guillemin-Kazhdan-80,Croke-Sharafutdinov-98,Paternain-Salo-Uhlmann-15} and \cite{Guillarmou-Paternain-Salo-Uhlmann-16} for the twisted version. Our aim is to obtain a pointwise version of this identity from the twisted Weitzenböck formula. Let us start with introducing the relevant curvature operators in our setting.

If $(E,\nabla^{E})$ is a vector bundle with metric connection, we denote by 
\[
R^{E} \in C^\infty(M, \Lambda^2\T^*M \otimes \End(E)),
\]
its curvature. Let $\mc{E}:=\pi^*E$ denote as before the pull-back of $E$ to $SM$ endowed with the pull-back connection $\nabla^{\mc{E}} := \pi^*\nabla^{E}$ and curvature 
\[
R^{\mc{E}}\in C^\infty(SM, \Lambda^2\T^*M \otimes \End(\mc{E})),
\]
satisfying $R^{\mc{E}}_{X,Y}(\pi^*\xi)=\pi^*(R^{E}_{X,Y}\xi)$ for all $X,Y\in\T M$ (identified with their horizontal lifts) and $\forall\xi\in C^\infty(M,E)$. Consider the vector bundle morphism $\mc{F}^{\mc{E}} :\mc{E}\to \mc{N} \otimes \mc{E}$ defined by:
\begin{equation}
\label{equation:twisted-curvature}
\langle \mc{F}^{\mc{E}}(\psi),w\otimes\psi' \rangle := \langle R^{\mc{E}}_{v,w}\psi,\psi'\rangle,
\end{equation}
for every $v \in SM, \  w \in \mc{N}_v $ and $\psi,\psi' \in \mc{E}_v$. The value of $\mc{F}^{\mc{E}}$ on pull-backs of sections of $E$ can be explicitly computed as
\begin{equation}
\label{equation:twisted-curvature1}
 \mc{F}^{\mc{E}}(\pi^*\xi)=\sum_i \e_i^\bot\otimes\pi^*( R_{v,\e_i}^E\xi),
\end{equation}
where $(\e_i)$ is a local orthonormal frame. We also define a vector bundle morphism $\mc{R} :\mc{N} \otimes\mc{E}\to \mc{N} \otimes \mc{E}$ by:
\begin{equation}
\label{equation:twisted-curvature-R}
 \mc{R}(w\otimes\psi) :=( R_{w,v}v)\otimes\psi,
\end{equation}
for every $v \in SM, \  w \in \mc{N}_v $ and $\psi\in \mc{E}_v$, where $R$ is the Riemann curvature tensor of $(M,g)$.

We will now give the relations between the operators $\mc{R}$ and $\mc{F}^{\mc{E}}$ on one side, and $q(R)$ and $R^{E}$ on the other side.

\begin{lemma}
\label{lemma:link-f}
For every $K\in C^\infty(M,\Sym^k_0 \T M)$ and $\xi  \in C^\infty(M,E)$, the following relations hold:
\begin{eqnarray}
\label{36} \nabla_\V^*  \mc{R} \nabla_\V \pi_k^* (K\otimes\xi)&=&\pi_{k}^*((q(R)K)\otimes\xi),\\
\label{37} \nabla_\V^*  \mc{F}^{\mc{E}} \pi_k^* (K\otimes\xi)&=&\tfrac12\pi_{k}^*\left(\sum_{i, j}(\e_i\wedge \e_j)_*K\otimes R^E_{\e_i,\e_j}\xi\right).
\end{eqnarray}
\end{lemma}

\begin{proof}Using \eqref{30} we compute at some $v\in SM$ the left-hand side of \eqref{36} as:
\begin{equation*}
\begin{split}
 \nabla_\V^*  \mc{R} \nabla_\V \pi_k^* (K\otimes\xi) &=\sum_i \nabla_\V^* \mc{R}\left(\pi_\mathcal{N}\left(\pi_{k-1}^*( \e_i\lrcorner K)\otimes(\e_i\otimes \xi)\right)\right) \\
& =\sum_{i} \nabla_\V^* \mc{R}\left(\pi_{k-1}^*( \e_i\lrcorner K)\pi_\mathcal{N}(\e_i)\otimes \pi^*\xi\right) \\
& =\sum_{i} \nabla_\V^*\left(\pi_{k-1}^*( \e_i\lrcorner K)R_{\e_i,v}v\otimes \pi^*\xi\right) \\
&=\sum_{i, j, l} \nabla_\V^*\left((\pi_1^*\e_j)(\pi_1^*\e_l)\pi_{k-1}^*( \e_i\lrcorner K)R_{\e_i,\e_j}\e_l\otimes \pi^*\xi\right)\\
&=\sum_{i, j, l}\nabla_\V^*\left(\pi_{k+1}^*(\e_j\cdot \e_l\cdot( \e_i\lrcorner K))R_{\e_i,\e_j}\e_l\otimes \pi^*\xi\right).
\end{split}
\end{equation*}
Using \eqref{31} we can rewrite this last sum as
$$-\sum_{i, j, l}\pi_k^*\left(R_{\e_i,\e_j}\e_l\lrcorner(\e_j\cdot \e_l\cdot( \e_i\lrcorner K))\otimes \xi\right) + (k+1)\sum_{i, j, l}\pi^*_{k+2}\left(R_{\e_i,\e_j}\e_l\cdot \e_j\cdot \e_l\cdot( \e_i\lrcorner K)\otimes \xi\right).$$
By Lemma \ref{qr1} the first summand is equal to $\pi_{k}^*((q(R)K)\otimes\xi)$. The second summand vanishes since $\sum_lR_{\e_i,\e_j}\e_l\cdot \e_l = 0$. This proves \eqref{36}. Similarly, using \eqref{equation:twisted-curvature1} we compute at $v\in SM$:
\begin{equation*}
\begin{split}
 \nabla_\V^* \mc{F}^{\mc{E}} \pi_k^* (K\otimes\xi) & =\sum_i \nabla_\V^*\left(\pi_k^* (K) (\e_i^\bot\otimes \pi^*(R_{v,\e_i}^E\xi))\right) \\
 &=\sum_i \nabla_\V^*\left(\pi_k^* (K) (\e_i\otimes \pi^*(R_{v,\e_i}^E\xi))\right)\\
&=\sum_{i, j} \nabla_\V^*\left(\pi_k^* (K)\pi_1^*(\e_j)(\e_i\otimes \pi^*(R_{\e_j,\e_i}^E\xi))\right) \\
&=\sum_{i, j} \nabla_\V^*\left(\pi_{k+1}^* (\e_j\cdot K)(\e_i\otimes \pi^*(R_{\e_j,\e_i}^E\xi))\right)\\
&\stackrel{\eqref{31}}{=}-\sum_{i, j}\pi_{k}^*(\e_i\lrcorner(\e_j\cdot K))\pi^*(R_{\e_j,\e_i}^E\xi) +(k+1)\sum_{i, j}\pi_{k+2}^* (\e_i\cdot \e_j \cdot K)\pi^*(R_{\e_j,\e_i}^E\xi).
\end{split}
\end{equation*}
The second summand vanishes because of the skew-symmetry of $R_{\e_j,\e_i}^E$ in $i$ and $j$, whereas the first summand is equal to 
\begin{align*}
    -\sum_{i, j} \pi_{k}^*\left(\e_i\lrcorner(\e_j\cdot K)\otimes R_{\e_j,\e_i}^E\xi\right) &= \sum_{i, j}\pi_{k}^*\left(\e_j\cdot(\e_i\lrcorner K)\otimes R_{\e_i,\e_j}^E\xi\right)\\ 
    &= \tfrac12\sum_{i, j} \pi_{k}^*\left((\e_i\wedge \e_j)_* K\otimes R_{\e_i,\e_j}^E\xi\right).
\end{align*}

\end{proof}

Combining \eqref{qrv} with Lemma \ref{lemma:link-f}, we obtain for every section of $\Sym^k_0 \T M\otimes E$:
\begin{equation}
\label{eq:qr}
\pi_k^*q(R)^E=(\nabla_\V^*  \mc{R}  \nabla_\V+ \nabla_\V^*  \mc{F}^{\mc{E}}) \pi_k^*.
\end{equation}
Then, using Lemma \ref{lemma:link-d} we compute for every section of $\Sym^k_0 \T M\otimes E$:
\begin{equation}
\label{equation:d*d}
\pi_k^*\dd_0^*\dd_0=-(n+2k)\X_-\pi_{k+1}^*\dd_0=-(n+2k)\X_-\X_+\pi_{k}^*,
\end{equation}
and similarly
\begin{equation}
\label{equation:dd*}
\pi_k^*\dd_0\dd_0^*=\X_+\pi_{k-1}^*\dd_0^*=-(n+2k-2)\X_+\X_-\pi_{k}^*.
\end{equation}
Finally, by \eqref{equation:z} we obtain
\begin{equation}
\label{p3}
\pi_k^*(P^E_3)^*P^E_3=Z_k^*\pi_{k}^*P^E_3=Z_k^*Z_k\pi_{k}^*.
\end{equation}
Altogether, we obtain the following:

\begin{pro}[Pointwise Localized Pestov identity]
\label{proposition:plpi}
On $C^\infty(M, \Omega^k\otimes E) \subset C^\infty(SM, \mc{E})$, the following relation holds:
 \begin{equation}\label{twistpestov}
  \nabla_\V^*  \mc{R}  \nabla_\V + \nabla_\V^*  \mc{F}^{\mc{E}}= \tfrac{k(n+2k)}{k+1} \X_-\X_+  \,- \, \tfrac{(n+k-2)(n+2k-4)}{(n+k-3)}\,\X_+\X_-\,+ \,Z_k^*Z_k.
\end{equation}
\end{pro}

\begin{proof}
  Every section of $\Omega^k\otimes E$ can be written as $\pi_k^*\Psi$ for some twisted symmetric tensor $\Psi\in C^\infty(M,\Sym^k_0 \T M\otimes E)$. Then the twisted Weitzenböck formula (Proposition \ref{tw}) together with \eqref{eq:qr}--\eqref{p3} gives directly \eqref{twistpestov}.
\end{proof}

Applying \eqref{twistpestov} to $\Psi \in C^\infty(M,\Omega^k \otimes E)$, pairing with $\Psi$ and then integrating over $SM$ with respect to the Liouville measure, we retrieve the localized Pestov identity in its integrated version \cite[Lemma 2.3]{Cekic-Lefeuvre-Moroianu-Semmelmann-21}.
\bibliographystyle{alpha}

\bibliography{Biblio}

\end{document}